\documentclass[fleqn, 11pt]{article}
\usepackage{amsmath, amssymb, amsthm, graphicx}
\usepackage{hyperref}
\theoremstyle{definition}
\newtheorem{theorem}{Theorem}[section]
\newtheorem{criterion}{Criterion}[section]

\newtheorem{lemma}[theorem]{Lemma}

\newtheorem{Remark}[theorem]{Remark}

\newtheorem{Example}[theorem]{Example}
\newtheorem{Question}[theorem]{Question}

\numberwithin{equation}{section}

\begin{document}

\title{Positive elliptic-elliptic rotopulsators on Clifford tori of nonconstant size project onto regular polygons}


\author{Pieter Tibboel\\
Department of Mathematical Sciences\\
Xi'an Jiaotong-Liverpool University\\
Suzhou, China\\
Pieter.Tibboel@xjtlu.edu.cn}


\maketitle
\begin{abstract}
  Let $q_{1}$,...,$q_{n}$ be the position vectors of the point masses of the curved $n$-body problem. Consider any positive elliptic-elliptic rotopulsator solution $q_{i}^{T}=(r\cos{(\theta+\alpha_{i})},r\sin{(\theta+\alpha_{i})},\rho\cos{(\phi+\beta_{i})},\rho\sin{(\phi+\beta_{i})})$, $i\in\{1,...,n\}$, where $\alpha_{1},...,\alpha_{n},\beta_{1},...,\beta_{n}\in [0,2\pi)$ are constants, $\phi$, $\theta$, $r$ and $\rho$ are twice-differentiable, continuous, nonconstant functions, $r^{2}+\rho^{2}=1$, $r\geq 0$ and $\rho\geq 0$. We prove that the if the configuration of the point masses is of nonconstant size, the configuration of the vectors
  $(r\cos{(\theta+\alpha_{i})},r\sin{(\theta+\alpha_{i})})^{T}$ is a regular polygon, as is the configuration of the vectors $(\rho\cos{(\phi+\beta_{i})},\rho\sin{(\phi+\beta_{i})})^{T}$, $i\in\{1,...,n\}$.
\end{abstract}

\section{Introduction}
    Let $\sigma=\pm 1$. The $n$-body problem in spaces of constant Gaussian curvature, or curved $n$-body problem for short, is the problem of finding the dynamics of point masses \begin{align*}q_{1},...,\textrm{ }q_{n}\in\mathbb{M}_{\sigma}^{3}=\{(x_{1},x_{2},x_{3},x_{4})^{T}\in\mathbb{R}^{4}|x_{1}^{2}+x_{2}^{2}+x_{3}^{2}+\sigma x_{4}^{2}=\sigma\},\end{align*} with respective masses $m_{1}>0$,..., $m_{n}>0$, determined by the system of differential equations
  \begin{align}\label{EquationsOfMotion Curved}
   \ddot{q}_{i}=\sum\limits_{j=1,\textrm{ }j\neq i}^{n}\frac{m_{j}(q_{j}-\sigma(q_{i}\odot q_{j})q_{i})}{(\sigma -\sigma(q_{i}\odot q_{j})^{2})^{\frac{3}{2}}}-\sigma(\dot{q}_{i}\odot\dot{q}_{i})q_{i},\textrm{ }i\in\{1,...,\textrm{ }n\},
  \end{align}
  where for $x$, $y\in\mathbb{M}_{\sigma}^{3}$  the product $\cdot\odot\cdot$ is defined as
  \begin{align*}
    x\odot y=x_{1}y_{1}+x_{2}y_{2}+x_{3}y_{3}+\sigma x_{4}y_{4}.
  \end{align*}
  The study of the curved $n$-body problem has applications to for example geometric mechanics, Lie groups and algebras, non-Euclidean and differential geometry and stability theory, the theory of polytopes and topology (see for example \cite{D6}) and for $n=2$ goes back as far as the 1830s (see \cite{BM}, \cite{BMK}, \cite{CRS}, \cite{DK}, \cite{DPS1}, \cite{DPS2}, \cite{DPS3} and \cite{KH} for a historical overview and recent results). However, the first paper giving an explicit $n$-body problem in spaces of constant Gaussian curvature for general $n\geq 2$ was published in 2008 by Diacu, P\'erez-Chavela and Santoprete (see \cite{DPS1}, \cite{DPS2} and \cite{DPS3}). This breakthrough then gave rise to further results for the $n\geq 2$ case in \cite{DeDZ}, \cite{D1}--\cite{DK}, \cite{DPo}, \cite{DT}, \cite{PS}--\cite{ZZ} and the references therein.

  Rotopulsators are solutions to (\ref{EquationsOfMotion Curved}) for which the configuration of the point massses may only rotate or change size, but retains its shape over time. They were first introduced by Diacu and Kordlou in \cite{DK} and can be divided into five classes, two for the positive curvature case ($\sigma=1$) and three for the negative curvature case ($\sigma=-1$). Positive elliptic-elliptic rotopulsators are one of the two possible types for the positive curvature case and can be defined as follows: Let
  \begin{align*}
    R(x)=\begin{pmatrix}
      \cos{x} & -\sin{x} \\
      \sin{x} & \cos{x}
    \end{pmatrix}.
  \end{align*}
  If we write $q_{i}=(q_{i1},q_{i2},q_{i3},q_{i4})^T$, $i\in\{1,...,n\}$, then we call $q_{1},...,q_{n}$ a \textit{positive elliptic-elliptic rotopulsator} if there exist nonnegative scalar functions $r_{i}$, $\rho_{i}$ for which $r_{i}^{2}+\rho_{i}^{2}=1$, scalar functions $\theta$, $\phi$ and constants $\alpha_{i}$, $\beta_{i}\in\mathbb{R}$, such that
  \begin{align}
    &\begin{pmatrix}
      q_{i1}\\ q_{i2}
    \end{pmatrix}=r_{i}R(\theta+\alpha_{i})\begin{pmatrix}
      1 \\ 0
    \end{pmatrix}\textrm{ and }\begin{pmatrix}
      q_{i3}\\ q_{i4}
    \end{pmatrix}=\rho_{i}R(\phi+\beta_{i})\begin{pmatrix}
      1 \\ 0
    \end{pmatrix}\label{Rotopulsator identities}.
  \end{align}
  Another definition we will need is the definition of a Clifford torus: Let $a>0$ and $b>0$. By a Clifford torus, we mean any set \begin{align*}\{(x_{1},x_{2},x_{3},x_{4})\in\mathbb{R}^{4}|x_{1}^{2}+x_{2}^{2}=a^{2},\textrm{ }x_{3}^{2}+x_{4}^{2}=b^{2}\}\end{align*}
  Finally, before we get to this paper's main result, a remark on terminology is needed: Throughout this paper, if $v_{1}$, $v_{2}$,..., $v_{n}\in\mathbb{R}^{2}$ are $n$ not necessarily distinct vectors, then if among these vectors there are $k$ distinct vectors $v_{j_{1}}$,...,$v_{j_{k}}$, such that for all other vectors $v_{i}$, $i\in\{1,...,n\}\backslash\{j_{1},...,j_{k}\}$ we have that there is an $l\in\{1,...,k\}$ such that $v_{i}=v_{j_{l}}$, then we say that the vectors $v_{1}$,...,$v_{n}$ represent the vertices of a polygon with $k$ vertices. With that said, we will continue on the topic of positive elliptic-elliptic rotopulsators:
  In \cite{T2} it was proven that if the $\beta_{i}$ are all equal and the $r_{i}$ and $\rho_{i}$ are independent of $i$ and not constant, then the configuration of the point masses has to be a regular polygon and in \cite{T6} it was shown that in that case all masses have to be equal. For the case that the $r_{i}$ and $\rho_{i}$ are independent of $i$ and not constant and the $\beta_{i}$ are not necessarily equal, almost nothing is known, except that in \cite{DK} for $n=3$ nonexistence was proven for Lagrangian configurations and in \cite{DT}, for $n=4$, nonexistence was proven for rectangular configurations and all masses equal. In this paper we will prove for general $n$ that positive elliptic-elliptic rotopulsators on Clifford tori of nonconstant size project onto regular polygons. Specifically:
  \begin{theorem}\label{Main Theorem 1}
    Let $q_{1},...,q_{n}$ be a positive elliptic-elliptic rotopulsator solution of (\ref{EquationsOfMotion Curved}) with $r_{i}$ and $\rho_{i}$ independent of $i$ and not constant and the size of the configuration not fixed. Then both the $(q_{i1},q_{i2})^{T}$, $i\in\{1,...,n\}$ and the $(q_{i3},q_{i4})^{T}$, $i\in\{1,...,n\}$ represent vertices of a regular polygon.
  \end{theorem}
  We will now first formulate a criterion and lemmas needed to prove Theorem~\ref{Main Theorem 1} in section~\ref{Section Background Theory}, after which we will prove Theorem~\ref{Main Theorem 1} in section~\ref{Section proof of main theorem 1}.
  \section{Background theory}\label{Section Background Theory}
  Throughout this paper we will use the notation introduced in the previous section for rotopulsators. Additionally, we will adopt the notation $\langle\cdot,\cdot\rangle$ for the Euclidean inner product. To prove Theorem~\ref{Main Theorem 1}, ideally we would like to repeat the main argument used to prove that for positive elliptic-elliptic rotopulsators for which $\beta_{i}=\beta_{j}$ for all $i$, $j\in\{1,...,n\}$ the configuration of the point masses has to be a regular polygon (see \cite{T}, \cite{T2}).
  The problem with positive elliptic-elliptic rotopulsators for which $r_{i}$ and $\rho_{i}$ are independent of $i$, $i\in\{1,...,n\}$, is, as we will see, that there is a possibility that there are $i$, $j\in\{1,...,n\}$, $i\neq j$ for which $\cos{(\alpha_{j}-\alpha_{i})}-\cos{(\beta_{j}-\beta_{i})}=0$. To illustrate that, we will start by formulating a lemma and a criterion for the existence of positive elliptic-elliptic rotopulsators we need. These were proven for general $r_{i}$ and $\rho_{i}$, $i\in\{1,...,n\}$ in \cite{DK}, but as the proofs are very short, we will add them here:
  \begin{lemma}\label{Lemma 1}
    Let $q_{1},...,q_{n}$ be a positive elliptic-elliptic rotopulsator solution of (\ref{EquationsOfMotion Curved}) with $r_{i}=r$ and $\rho_{i}=\rho$ independent of $i$. Then $2\dot{r}\dot{\theta}+r\ddot{\theta}=0$ and $2\dot{\rho}\dot{\phi}+\rho\ddot{\phi}=0$.
  \end{lemma}
  \begin{proof}
    It was proven in \cite{D3} that $\sum\limits_{j=1}^{n}m_{j}q_{j}\wedge\dot{q}_{j}=\mathbf{0}$ where $\wedge$ represents the wedge product and $\mathbf{0}$ is the zero bivector. If $e_{1}$, $e_{2}$, $e_{3}$ and $e_{4}$ are the standard basis vectors of $\mathbb{R}^{4}$, then
    \begin{align*}
      0e_{1}\wedge e_{2}=\sum\limits_{j=1}^{n}m_{j}(q_{j1}\ddot{q}_{j2}-q_{j2}\ddot{q}_{j1})e_{1}\wedge e_{2}.
    \end{align*}
    Using (\ref{Rotopulsator identities}), $q_{j1}\ddot{q}_{j2}-q_{j2}\ddot{q}_{j1}$ can be rewritten as $2\dot{r}\dot{\theta}+r\ddot{\theta}$, so $0=\sum\limits_{j=1}^{n}m_{j}(2\dot{r}\dot{\theta}+r\ddot{\theta})$, which gives that $2\dot{r}\dot{\theta}+r\ddot{\theta}=0$. Repeating this argument replacing $e_{1}$ and $e_{2}$ with $e_{3}$ and $e_{4}$ and $q_{j1}$ and $q_{j2}$ with $q_{j3}$ and $q_{j4}$ gives that $2\dot{\rho}\dot{\phi}+\rho\ddot{\phi}=0$.
  \end{proof}
  \begin{criterion}\label{Criterion}
    Let $q_{1},...,q_{n}$ be a positive elliptic-elliptic rotopulsator solution of (\ref{EquationsOfMotion Curved}) with $r_{i}=r$ and $\rho_{i}=\rho$ independent of $i$. Then
    \begin{align}
      0 &=\sum\limits_{j=1,\textrm{ }j\neq i}^{n}\frac{m_{j}\sin{(\alpha_{j}-\alpha_{i})}}{(1-(\cos{(\beta_{j}-\beta_{i})}+r^{2}(\cos{(\alpha_{j}-\alpha_{i})}-\cos{(\beta_{j}-\beta_{i})}))^{2})^{\frac{3}{2}}}, \label{Crit1}\\
      0 &=\sum\limits_{j=1,\textrm{ }j\neq i}^{n}\frac{m_{j}\sin{(\beta_{j}-\beta_{i})}}{(1-(\cos{(\beta_{j}-\beta_{i})}+r^{2}(\cos{(\alpha_{j}-\alpha_{i})}-\cos{(\beta_{j}-\beta_{i})}))^{2})^{\frac{3}{2}}}\textrm{ and}\label{Crit2}\\
      \delta &=r\rho^{2}\sum\limits_{j=1,\textrm{ }j\neq i}^{n}\frac{m_{j}(\cos{(\alpha_{j}-\alpha_{i})}-\cos{(\beta_{j}-\beta_{i})})}{(1-(\cos{(\beta_{j}-\beta_{i})}+r^{2}(\cos{(\alpha_{j}-\alpha_{i})}-\cos{(\beta_{j}-\beta_{i})}))^{2})^{\frac{3}{2}}},\label{Crit3}
    \end{align}
    where $\delta=\ddot{r}+r\rho^{2}((\dot{\phi})^{2}-(\dot{\theta})^{2})+r\left(\frac{\dot{r}}{\rho}\right)^{2}$.
  \end{criterion}
  \begin{proof}
    Inserting (\ref{Rotopulsator identities}) into (\ref{EquationsOfMotion Curved}), writing out the expressions for the first two coordinates of the vectors on both sides of (\ref{EquationsOfMotion Curved}) and multiplying both sides of the resulting equations from the left with $R(\theta+\alpha_{i})^{-1}$ and using that $r^{2}+\rho^{2}=1$, we get that
    \begin{align}
    &(\ddot{r}-r\dot{\theta}^{2})\begin{pmatrix}
     1\\ 0
   \end{pmatrix}+(2\dot{r}\dot{\theta}+r\ddot{\theta})\begin{pmatrix}
     0\\1
   \end{pmatrix}\nonumber\\
   &=r\sum\limits_{j=1,\textrm{ }j\neq i}^{n}\frac{m_{j}\begin{pmatrix}
     \rho^{2}(\cos{(\alpha_{j}-\alpha_{i})}-\cos{(\beta_{j}-\beta_{i})})\\
     \sin{(\alpha_{j}-\alpha_{i})}
   \end{pmatrix}}{(1-(\cos{(\beta_{j}-\beta_{i})}+r^{2}(\cos{(\alpha_{j}-\alpha_{i})}-\cos{(\beta_{j}-\beta_{i})}))^{2})^{\frac{3}{2}}}\nonumber\\
   &-r(\dot{r}^{2}+r^{2}\dot{\theta}^{2}+\dot{\rho}^{2}+\rho^{2}\dot{\phi}^{2})\begin{pmatrix}
     1\\ 0
   \end{pmatrix}\label{CriterionStep1}.
  \end{align}
  Collecting terms for the first coordinates of the vectors on both sides of (\ref{CriterionStep1}) gives (\ref{Crit3}) and by Lemma~\ref{Lemma 1} collecting terms for the second coordinates of the vectors on both sides of (\ref{CriterionStep1}) gives (\ref{Crit1}). Repeating the argument used so far writing out the expressions for the third and fourth coordinates of the vectors on both sides of (\ref{EquationsOfMotion Curved}) instead then gives (\ref{Crit2}). The identity for the third coordinates can be rewritten as (\ref{Crit3}), which proves this criterion.
  \end{proof}
  We can now explain the problem touched upon at the start of this section in more detail: The idea behind our proof of Theorem~\ref{Main Theorem 1}, which is essentially the idea behind proving that the configurations of the point masses have to be regular polygons in \cite{T} and \cite{T2}, is to use that the terms in the identities of Criterion~\ref{Criterion} are linearly independent under certain conditions on the $\alpha_{i}$ and $\beta_{i}$, $i\in\{1,...,n\}$, if $r_{i}=r$ and $\rho_{i}=\rho$ are independent of $i$ and not constant. However, we immediately see that a prerequisite to using such a linear independence argument is that $\cos{(\alpha_{j}-\alpha_{i})}-\cos{(\beta_{j}-\beta_{i})}\neq 0$ for those terms. We will therefore now prove three lemmas leading to the fact that we can exclude the possibility of $\cos{(\alpha_{j}-\alpha_{i})}-\cos{(\beta_{j}-\beta_{i})}=0$ and then give the exact conditions for the terms of the sums on the right-hand side of (\ref{Crit1}), (\ref{Crit2}) and (\ref{Crit3}) to be linearly independent:
  \begin{lemma}\label{Lemma 3}
    Let $q_{1},...,q_{n}$ be a positive elliptic-elliptic rotopulsator solution of (\ref{EquationsOfMotion Curved}) with $r_{i}=r$ and $\rho_{i}=\rho$ independent of $i$ and not constant and the size of the configuration of the point masses not fixed. Then for all but possibly one $i\in\{1,...,n\}$ we have that \begin{align*}\cos{(\alpha_{j}-\alpha_{i})}-\cos{(\beta_{j}-\beta_{i})}\neq 0\textrm{ for at least one }j\in\{1,...,n\}\backslash\{i\}.\end{align*}
  \end{lemma}
  \begin{proof}
    If there are $i_{1}$, $i_{2}\in\{1,...,n\}$ such that \begin{align*}\cos{(\alpha_{j}-\alpha_{i})}-\cos{(\beta_{j}-\beta_{i})}=0\textrm{ for all }j\in\{1,...,n\},\textrm{ }i\in\{i_{1},i_{2}\},\end{align*} then
    \begin{align}\langle q_{i},q_{j}\rangle &=r^{2}\cos{(\alpha_{j}-\alpha_{i})}+\rho^{2}\cos{(\beta_{j}-\beta_{i})}=r^{2}\cos{(\alpha_{j}-\alpha_{i})}+\rho^{2}\cos{(\alpha_{j}-\alpha_{i})}\nonumber\\
    &=(r^{2}+\rho^{2})\cos{(\alpha_{j}-\alpha_{i})}=\cos{(\alpha_{j}-\alpha_{i})}\label{Angles}\end{align}
    for $i\in\{i_{1},i_{2}\}$,
    which means that the angles between $q_{j}$, $q_{i_{1}}$ and $q_{i_{2}}$ are constant. As the length of each position vector is $1$, we have that $q_{i_{1}}$ and $q_{i_{2}}$ are linearly dependent if and only if $q_{i_{1}}=-q_{i_{2}}$, which would mean that $1-\langle q_{i_{1}},q_{i_{2}}\rangle^{2}=0$, which would make (\ref{EquationsOfMotion Curved}) undefined, so $q_{i_{1}}$ and $q_{i_{2}}$ are linearly independent. As the lengths of $q_{i_{1}}$, $q_{i_{2}}$ and $q_{j}$ are all equal to $1$ and therefore fixed and the angles between $q_{i_{1}}$ and $q_{i_{2}}$, between $q_{i_{1}}$ and $q_{j}$ and between $q_{i_{2}}$ and $q_{j}$ are fixed by (\ref{Angles}), the configuration of the point masses $q_{i_{1}}$, $q_{i_{2}}$ and $q_{j}$ is fixed for any $j$, which means that the configuration of the point masses can only rotate, but otherwise stays unchanged over time, which contradicts that the size of the configuration of the point masses is not constant. This means by extension that for all but perhaps one $i\in\{1,...,n\}$, (\ref{Crit1}), (\ref{Crit2}) and (\ref{Crit3}) have to have at least one term for which $\cos{(\alpha_{j}-\alpha_{i})}-\cos{(\beta_{j}-\beta_{i})})\neq 0$.
  \end{proof}
  Next we will prove a lemma that serves as a check on restrictions on the $q_{i}$, $i\in\{1,...,n\}$ due to the $q_{i}$ being the position vectors of a positive elliptic-elliptic rotopulsator for which the $r_{i}$ and $\rho_{i}$ are independent of $i$ and not constant:
  \begin{lemma}\label{Lemma 4}
    Let $q_{1},...,q_{n}$ be a positive elliptic-elliptic rotopulsator solution of (\ref{EquationsOfMotion Curved}) with $r_{i}=r$ and $\rho_{i}=\rho$ independent of $i$ and not constant. Let $i_{1}$, $i_{2}$, $i_{3}\in\{1,...,n\}$. Then
    \begin{align}
         &\left((1+B_{i_{1}i_{2},i_{1}i_{3}})+r^{2}(A_{i_{1}i_{2},i_{1}i_{3}}-B_{i_{1}i_{2},i_{1}i_{3}})\right)^{2}\nonumber\\
         &=4(\cos^{2}{\gamma_{i_{1}i_{2},i_{1}i_{3}}})\left(1-\cos{(\beta_{i_{2}}-\beta_{i_{1}})}-r^{2}(\cos{(\alpha_{i_{2}}-\alpha_{i_{1}})}-\cos{(\beta_{i_{2}}-\beta_{i_{1}})})\right)\nonumber\\
         &\cdot\left(1-\cos{(\beta_{i_{3}}-\beta_{i_{1}})}-r^{2}(\cos{(\alpha_{i_{3}}-\alpha_{i_{1}})}-\cos{(\beta_{i_{3}}-\beta_{i_{1}})})\right)\label{AlphaBetaCases1}
    \end{align}
    and
    \begin{align}
         &\left((1+A_{i_{1}i_{2},i_{1}i_{3}})+\rho^{2}(B_{i_{1}i_{2},i_{1}i_{3}}-A_{i_{1}i_{2},i_{1}i_{3}})\right)^{2}\nonumber\\
         &=4(\cos^{2}{\gamma_{i_{1}i_{2},i_{1}i_{3}}})\left(1-\cos{(\alpha_{i_{2}}-\alpha_{i_{1}})}-\rho^{2}(\cos{(\beta_{i_{2}}-\beta_{i_{1}})}-\cos{(\alpha_{i_{2}}-\alpha_{i_{1}})})\right)\nonumber\\
         &\cdot\left(1-\cos{(\alpha_{i_{3}}-\alpha_{i_{1}})}-\rho^{2}(\cos{(\beta_{i_{3}}-\beta_{i_{1}})}-\cos{(\alpha_{i_{3}}-\alpha_{i_{1}})})\right),\label{AlphaBetaCases2}
    \end{align}
    where $\gamma_{i_{1}i_{2},i_{1}i_{3}}$ is the constant angle between $q_{i_{2}}-q_{i_{1}}$ and $q_{i_{3}}-q_{i_{1}}$, \begin{align*}A_{i_{1}i_{2},i_{1}i_{3}}=\cos{(\alpha_{i_{3}}-\alpha_{i_{2}})}-\cos{(\alpha_{i_{2}}-\alpha_{i_{1}})}-\cos{(\alpha_{i_{3}}-\alpha_{i_{1}})}\end{align*} and \begin{align*}B_{i_{1}i_{2},i_{1}i_{3}}=\cos{(\beta_{i_{3}}-\beta_{i_{2}})}-\cos{(\beta_{i_{2}}-\beta_{i_{1}})}-\cos{(\beta_{i_{3}}-\beta_{i_{1}})}. \end{align*}
  \end{lemma}
  \begin{proof}
    Because the $q_{i}$, $i\in\{1,...,n\}$ are the position vectors of the point masses of a rotopulsator, the shape of the configuration of the point masses has to remain unchanged over time. This means in particular that
    \begin{align*}
      \langle q_{i_{2}}-q_{i_{1}}, q_{i_{3}}-q_{i_{1}}\rangle=\|q_{i_{2}}-q_{i_{1}}\|\|q_{i_{3}}-q_{i_{1}}\|\cos{\gamma_{i_{1}i_{2},i_{1}i_{3}}},
    \end{align*}
    which means that
    \begin{align}\label{Go-time}
      \langle q_{i_{2}}-q_{i_{1}}, q_{i_{3}}-q_{i_{1}}\rangle^{2}=\|q_{i_{2}}-q_{i_{1}}\|^{2}\|q_{i_{3}}-q_{i_{1}}\|^{2}\cos^{2}{\gamma_{i_{1}i_{2},i_{1}i_{3}}}.
    \end{align}
    Writing out the left-hand side of (\ref{Go-time}) gives
    \begin{align*}
      \langle q_{i_{2}}-q_{i_{1}}, q_{i_{3}}-q_{i_{1}}\rangle^{2}&=\left(1-\langle q_{i_{1}}, q_{i_{2}}+q_{i_{3}}\rangle+\langle q_{i_{2}},q_{i_{3}}\rangle\right)^{2}\nonumber\\
      &=\left(1-\left(r^{2}(\cos{(\alpha_{i_{2}}-\alpha_{i_{1}})}+\cos{(\alpha_{i_{3}}-\alpha_{i_{1}})})\right.\right.\\
      &+\left.\rho^{2}(\cos{(\beta_{i_{2}}-\beta_{i_{1}})}+\cos{(\beta_{i_{3}}-\beta_{i_{1}})})\right)\\
      &\left.\left.+r^{2}\cos{(\alpha_{i_{3}}-\alpha_{i_{2}})} +\rho^{2}\cos{(\beta_{i_{3}}-\beta_{i_{2}})}\right)\right)^{2},
    \end{align*}
    which, as $r^{2}+\rho^{2}=1$, gives
    \begin{align}\label{AlphaBetaCases1Preliminary}
      \langle q_{i_{2}}-q_{i_{1}}, q_{i_{3}}-q_{i_{1}}\rangle^{2}=\left((1+B_{i_{1}i_{2},i_{1}i_{3}})+r^{2}(A_{i_{1}i_{2},i_{1}i_{3}}-B_{i_{1}i_{2},i_{1}i_{3}})\right)^{2}.
    \end{align}
    Writing out the right-hand side of (\ref{Go-time}) gives
    \begin{align*}
      &\|q_{i_{2}}-q_{i_{1}}\|^{2}\|q_{i_{3}}-q_{i_{1}}\|^{2}\cos^{2}{\gamma_{i_{1}i_{2},i_{1}i_{3}}}\\
      &=4(1-\langle q_{i_{1}},q_{i_{2}}\rangle)(1-\langle q_{i_{1}},q_{i_{3}}\rangle)\cos^{2}{\gamma_{i_{1}i_{2},i_{1}i_{3}}}\\
      &=4(1-(r^{2}\cos{(\alpha_{i_{2}}-\alpha_{i_{1}})}+\rho^{2}\cos{(\beta_{i_{2}}-\beta_{i_{1}})}))\\
      &\cdot(1-(r^{2}\cos{(\alpha_{i_{3}}-\alpha_{i_{1}})}+\rho^{2}\cos{(\beta_{i_{3}}-\beta_{i_{1}})}))\cos^{2}{\gamma_{i_{1}i_{2},i_{1}i_{3}}},
    \end{align*}
    which, using again that $r^{2}+\rho^{2}=1$, gives
    \begin{align}
      &\|q_{i_{2}}-q_{i_{1}}\|^{2}\|q_{i_{3}}-q_{i_{1}}\|^{2}\cos^{2}{\gamma_{i_{1}i_{2},i_{1}i_{3}}}\nonumber\\
      &=4(\cos^{2}{\gamma_{i_{1}i_{2},i_{1}i_{3}}})\left(1-\cos{(\beta_{i_{2}}-\beta_{i_{1}})-r^{2}(\cos{(\alpha_{i_{2}}-\alpha_{i_{1}})}-\cos{(\beta_{i_{2}}-\beta_{i_{1}})})}\right)\nonumber\\
      &\cdot\left(1-\cos{(\beta_{i_{3}}-\beta_{i_{1}})-r^{2}(\cos{(\alpha_{i_{3}}-\alpha_{i_{1}})}-\cos{(\beta_{i_{3}}-\beta_{i_{1}})})}\right).\label{AlphaBetaCases1Preliminary2}
    \end{align}
    Combining (\ref{AlphaBetaCases1Preliminary}) and (\ref{AlphaBetaCases1Preliminary2}) then gives (\ref{AlphaBetaCases1}) and interchanging the roles of the $\alpha$s and the $\beta$s, of the $A$s and the $B$s and of $r$ and $\rho$ gives (\ref{AlphaBetaCases2}).
  \end{proof}
  \begin{lemma}
    Let $q_{1},...,q_{n}$ be a positive elliptic-elliptic rotopulsator solution of (\ref{EquationsOfMotion Curved}) with $r_{i}=r$ and $\rho_{i}=\rho$ independent of $i$ and not constant and a configuration of nonconstant size. Then for all $i$, $j\in\{1,...,n\}$, $i\neq j$, we have that $\cos{(\alpha_{j}-\alpha_{i})}-\cos{(\beta_{j}-\beta_{i})}\neq 0$ and for all $k\in\{1,...,n\}\backslash\{i,j\}$, we have that \begin{align}\label{The Ultimate Identity1}
      \frac{1-\cos{(\beta_{i}-\beta_{j})}}{\cos{(\alpha_{i}-\alpha_{j})}-\cos{(\beta_{i}-\beta_{j})}}&=\frac{1-\cos{(\beta_{k}-\beta_{j})}}{\cos{(\alpha_{k}-\alpha_{j})}-\cos{(\beta_{k}-\beta_{j})}}.
    \end{align}
    and
    \begin{align}\label{The Ultimate Identity2}
      \frac{1-\cos{(\alpha_{i}-\alpha_{j})}}{\cos{(\alpha_{i}-\alpha_{j})}-\cos{(\beta_{i}-\beta_{j})}}&=\frac{1-\cos{(\alpha_{k}-\alpha_{j})}}{\cos{(\alpha_{k}-\alpha_{j})}-\cos{(\beta_{k}-\beta_{j})}}.    \end{align}
  \end{lemma}
  \begin{proof}
    Note that in (\ref{AlphaBetaCases1}), if $4\cos^{2}{\gamma_{i_{1}i_{2},i_{1}i_{3}}}\neq 0$, as $r$ is not constant, we have that the left-hand side of (\ref{AlphaBetaCases1}) is a polynomial in terms of $r^{2}$ of degree two with a double root and the right-hand side of (\ref{AlphaBetaCases1}) is a polynomial of degree two in terms of $r^{2}$ that therefore also needs to have a double root, which means that
    \begin{align*}
      \frac{1-\cos{(\beta_{i_{2}}-\beta_{i_{1}})}}{\cos{(\alpha_{i_{2}}-\alpha_{i_{1}})}-\cos{(\beta_{i_{2}}-\beta_{i_{1}})}}=\frac{1-\cos{(\beta_{i_{3}}-\beta_{i_{1}})}}{\cos{(\alpha_{i_{3}}-\alpha_{i_{1}})}-\cos{(\beta_{i_{3}}-\beta_{i_{1}})}}
    \end{align*}
    if $\cos{(\alpha_{i_{2}}-\alpha_{i_{1}})}-\cos{(\beta_{i_{2}}-\beta_{i_{1}})}\neq 0$ and $\cos{(\alpha_{i_{3}}-\alpha_{i_{1}})}-\cos{(\beta_{i_{3}}-\beta_{i_{1}})}\neq 0$. If $\cos{(\alpha_{i_{2}}-\alpha_{i_{1}})}-\cos{(\beta_{i_{2}}-\beta_{i_{1}})}=0$, or $\cos{(\alpha_{i_{3}}-\alpha_{i_{1}})}-\cos{(\beta_{i_{3}}-\beta_{i_{1}})}=0$, then again by (\ref{AlphaBetaCases1}) we have that both $\cos{(\alpha_{i_{2}}-\alpha_{i_{1}})}-\cos{(\beta_{i_{2}}-\beta_{i_{1}})}=0$ and $\cos{(\alpha_{i_{3}}-\alpha_{i_{1}})}-\cos{(\beta_{i_{3}}-\beta_{i_{1}})}=0$, as otherwise the degrees of the polynomials on the left-hand side and the right-hand side of (\ref{AlphaBetaCases1}) would not match. Suppose that $\gamma_{i_{1}i_{2},i_{1}i_{3}}=\frac{\pi}{2}$. Then we repeat this argument for $\gamma_{i_{3}i_{2},i_{3}i_{1}}$ and $\gamma_{i_{2}i_{1},i_{2}i_{3}}$, which cannot be equal to $\frac{\pi}{2}$, as a triangle can have at most one right angle, instead, finding that either
    \begin{align*}
      \frac{1-\cos{(\beta_{i_{2}}-\beta_{i_{1}})}}{\cos{(\alpha_{i_{2}}-\alpha_{i_{1}})}-\cos{(\beta_{i_{2}}-\beta_{i_{1}})}}&=\frac{1-\cos{(\beta_{i_{3}}-\beta_{i_{1}})}}{\cos{(\alpha_{i_{3}}-\alpha_{i_{1}})}-\cos{(\beta_{i_{3}}-\beta_{i_{1}})}}\\
      &=\frac{1-\cos{(\beta_{i_{3}}-\beta_{i_{2}})}}{\cos{(\alpha_{i_{3}}-\alpha_{i_{2}})}-\cos{(\beta_{i_{3}}-\beta_{i_{2}})}},
    \end{align*}
    or that $\cos{(\alpha_{i_{k}}-\alpha_{i_{s}})}-\cos{(\beta_{i_{k}}-\beta_{i_{s}})}=0$ for all $k$, $s\in\{1,2,3\}$.
    As this last result holds true for every triangle of point masses, we have that if there are $i$, $j\in\{1,...,n\}$, $i\neq j$ for which $\cos{(\alpha_{j}-\alpha_{i})}-\cos{(\beta_{j}-\beta_{i})}=0$, then we have that $\cos{(\alpha_{j}-\alpha_{i})}-\cos{(\beta_{j}-\beta_{i})}=0$ for all $i$, $j\in\{1,...,n\}$. As by Lemma~\ref{Lemma 3} there are $i$, $j\in\{1,...,n\}$ for which  \begin{align*}\cos{(\alpha_{j}-\alpha_{i})}-\cos{(\beta_{j}-\beta_{i})}\neq 0,\end{align*} this means that $\cos{(\alpha_{j}-\alpha_{i})}-\cos{(\beta_{j}-\beta_{i})}\neq 0$ for all $i$, $j\in\{1,...,n\}$, $i\neq j$, so that means that (\ref{The Ultimate Identity1}) holds for all $i$, $j$, $k\in\{1,...,n\}$, $i\neq j$, $i\neq k$, $j\neq $. Repeating the argument so far after interchanging the roles of the $\alpha$s and the $\beta$s and $r$ and $\rho$ then proves that (\ref{The Ultimate Identity2}) holds for all $i$, $j$, $k\in\{1,...,n\}$, $i\neq j$, $i\neq k$, $j\neq k$ as well. This completes the proof.
  \end{proof}
  We can now conclude:
  \begin{lemma}\label{Lemma 2}
    Let $q_{1},...,q_{n}$ be a positive elliptic-elliptic rotopulsator solution of (\ref{EquationsOfMotion Curved}) with a configuration of nonconstant size, $r_{i}=r$ and $\rho_{i}=\rho$ independent of $i$ and not constant. Let $A_{j_{1}i_{1}}$ and $A_{j_{2}i_{2}}$ be nonzero constants, $i_{1}$, $i_{2}$, $j_{1}$, $j_{2}\in\{1,...,n\}$, $j_{1}\neq i_{1}$ and $j_{2}\neq i_{2}$. Then for terms
    \begin{align*}
      \frac{A_{j_{1}i_{1}}}{(1-(\cos{(\beta_{j_{1}}-\beta_{i_{1}})}+r^{2}(\cos{(\alpha_{j_{1}}-\alpha_{i_{1}})}-\cos{(\beta_{j_{1}}-\beta_{i_{1}})}))^{2})^{\frac{3}{2}}}
    \end{align*}
    to cancel out against terms
    \begin{align*}
      \frac{A_{j_{2}i_{2}}}{(1-(\cos{(\beta_{j_{2}}-\beta_{i_{2}})}+r^{2}(\cos{(\alpha_{j_{2}}-\alpha_{i_{2}})}-\cos{(\beta_{j_{2}}-\beta_{i_{2}})}))^{2})^{\frac{3}{2}}}
    \end{align*}
    in the sums in the right-hand sides of (\ref{Crit1}), (\ref{Crit2}) and (\ref{Crit3}), we need that \begin{align*}\cos{(\alpha_{j_{1}}-\alpha_{i_{1}})}=\cos{(\alpha_{j_{2}}-\alpha_{i_{2}})}\textrm{ and }\cos{(\beta_{j_{1}}-\beta_{i_{1}})}=\cos{(\beta_{j_{2}}-\beta_{i_{2}})}. \end{align*}
    \end{lemma}
  \begin{proof}
    Note that if $\cos{(\alpha_{j_{1}}-\alpha_{i_{1}})}\neq\cos{(\beta_{j_{1}}-\beta_{i_{1}})}$, then terms
    \begin{align*}
      \frac{A_{j_{1}i_{1}}}{(1-(\cos{(\beta_{j_{1}}-\beta_{i_{1}})}+r^{2}(\cos{(\alpha_{j_{1}}-\alpha_{i_{1}})}-\cos{(\beta_{j_{1}}-\beta_{i_{1}})}))^{2})^{\frac{3}{2}}}
    \end{align*}
    are linearly independent for different values of $i_{1}$, $j_{1}$ if and only if the roots of the polynomials $(1-(\cos{(\beta_{j_{1}}-\beta_{i})}+x(\cos{(\alpha_{j_{1}}-\alpha_{i})}-\cos{(\beta_{j_{1}}-\beta_{i})}))^{2})$ differ for different values of $j_{1}$. As the roots of \begin{align*}(1-(\cos{(\beta_{j_{1}}-\beta_{i_{1}})}+x(\cos{(\alpha_{j_{1}}-\alpha_{i_{1}})}-\cos{(\beta_{j_{1}}-\beta_{i_{1}})}))^{2})\end{align*} are \begin{align*}\frac{1-\cos{(\beta_{j_{1}}-\beta_{i_{1}})}}{\cos{(\alpha_{j_{1}}-\alpha_{i_{1}})}-\cos{(\beta_{j_{1}}-\beta_{i_{1}})}}\textrm{ and } -\frac{1+\cos{(\beta_{j_{1}}-\beta_{i_{1}})}}{\cos{(\alpha_{j_{1}}-\alpha_{i_{1}})}-\cos{(\beta_{j_{1}}-\beta_{i_{1}})}}\end{align*} if $\cos{(\alpha_{j_{1}}-\alpha_{i_{1}})}\neq\cos{(\beta_{j_{1}}-\beta_{i_{1}})}$, we have that terms
    \begin{align*}
      \frac{A_{j_{1}i_{1}}}{(1-(\cos{(\beta_{j_{1}}-\beta_{i_{1}})}+r^{2}(\cos{(\alpha_{j_{1}}-\alpha_{i_{1}})}-\cos{(\beta_{j_{1}}-\beta_{i_{1}})}))^{2})^{\frac{3}{2}}}
    \end{align*}
    can cancel out against terms
    \begin{align*}
      \frac{A_{j_{2}i_{2}}}{(1-(\cos{(\beta_{j_{2}}-\beta_{i_{2}})}+r^{2}(\cos{(\alpha_{j_{2}}-\alpha_{i_{2}})}-\cos{(\beta_{j_{2}}-\beta_{i_{2}})}))^{2})^{\frac{3}{2}}}
    \end{align*}
    if and only if
    \begin{align*}\frac{1-\cos{(\beta_{j_{1}}-\beta_{i_{1}})}}{\cos{(\alpha_{j_{1}}-\alpha_{i_{1}})}-\cos{(\beta_{j_{1}}-\beta_{i_{1}})}}=\frac{1-\cos{(\beta_{j_{2}}-\beta_{i_{2}})}}{\cos{(\alpha_{j_{2}}-\alpha_{i_{2}})}-\cos{(\beta_{j_{2}}-\beta_{i_{2}})}}\end{align*} and \begin{align*}-\frac{1+\cos{(\beta_{j_{1}}-\beta_{i_{1}})}}{\cos{(\alpha_{j_{1}}-\alpha_{i_{1}})}-\cos{(\beta_{j_{1}}-\beta_{i_{1}})}}=-\frac{1+\cos{(\beta_{j_{2}}-\beta_{i_{2}})}}{\cos{(\alpha_{j_{2}}-\alpha_{i_{2}})}-\cos{(\beta_{j_{2}}-\beta_{i_{2}})}},\end{align*}
    or
     \begin{align*}\frac{1-\cos{(\beta_{j_{1}}-\beta_{i_{1}})}}{\cos{(\alpha_{j_{1}}-\alpha_{i_{1}})}-\cos{(\beta_{j_{1}}-\beta_{i_{1}})}}=-\frac{1-\cos{(\beta_{j_{2}}-\beta_{i_{2}})}}{\cos{(\alpha_{j_{2}}-\alpha_{i_{2}})}-\cos{(\beta_{j_{2}}-\beta_{i_{2}})}}\end{align*} and \begin{align*}\frac{1+\cos{(\beta_{j_{1}}-\beta_{i_{1}})}}{\cos{(\alpha_{j_{1}}-\alpha_{i_{1}})}-\cos{(\beta_{j_{1}}-\beta_{i_{1}})}}=-\frac{1+\cos{(\beta_{j_{2}}-\beta_{i_{2}})}}{\cos{(\alpha_{j_{2}}-\alpha_{i_{2}})}-\cos{(\beta_{j_{2}}-\beta_{i_{2}})}}.\end{align*}
    The first of these possibilities is equivalent with $\cos{(\beta_{j_{1}}-\beta_{i_{1}})}=\cos{(\beta_{j_{2}}-\beta_{i_{2}})}$ and \begin{align*}
      \cos{(\alpha_{j_{1}}-\alpha_{i})}-\cos{(\beta_{j_{1}}-\beta_{i})}=\cos{(\alpha_{j_{2}}-\alpha_{i})}-\cos{(\beta_{j_{2}}-\beta_{i})},
    \end{align*}
    which, as
    \begin{align*}\frac{1-\cos{(\beta_{j}-\beta_{i})}}{\cos{(\alpha_{j}-\alpha_{i})}-\cos{(\beta_{j}-\beta_{i})}}&=\frac{1-\cos{(\alpha_{j}-\beta_{i})}+\cos{(\alpha_{j}-\beta_{i})}-\cos{(\alpha_{j}-\beta_{i})}}{\cos{(\alpha_{j}-\alpha_{i})}-\cos{(\beta_{j}-\beta_{i})}}\nonumber\\
    &=1+\frac{1-\cos{(\alpha_{j}-\beta_{i})}}{\cos{(\alpha_{j}-\alpha_{i})}-\cos{(\beta_{j}-\beta_{i})}},\end{align*}
    also means that $\cos{(\alpha_{j_{1}}-\alpha_{i_{1}})}=\cos{(\alpha_{j_{2}}-\alpha_{i_{2}})}$. \\
    The second of these possibilities can be disgarded by (\ref{The Ultimate Identity1}) and (\ref{The Ultimate Identity2}). This completes the proof.
  \end{proof}
  \section{Proof of Theorem~\ref{Main Theorem 1}}\label{Section proof of main theorem 1}
    We will assume that there are $i$, $j\in\{1,...,n\}$ such that $\alpha_{j}-\alpha_{i}\neq 0$. If this is not the case, then we switch the roles of the $\alpha$s and the $\beta$s. Additionally, let $S_{\alpha j}=\{s\in\{1,...,n\}|\alpha_{s}=\alpha_{j}\}$ and select $j_{1}$,...,$j_{k}\in\{1,...,n\}$ such that $S_{\alpha s_{1}}\cap S_{\alpha s_{2}}=\emptyset$ for $s_{1}$, $s_{2}\in\{j_{1},...,j_{k}\}$, $s_{1}\neq s_{2}$ and $\bigcup\limits_{u=1}^{k}S_{\alpha j_{u}}=\{1,...,n\}$, $0\leq\alpha_{j_{1}}<\alpha_{j_{2}}<...<\alpha_{j_{k}}<2\pi$ and $\alpha_{j_{2}}-\alpha_{j_{1}}\leq\alpha_{j_{u+1}}-\alpha_{j_{u}}$ for all $u\in\{1,...,k\}$, relabeling the point masses if necessary, where we define $\alpha_{j_{k+1}}=2\pi+\alpha_{j_{1}}$.  We will use a proof by contradiction. Assume that the vectors $(q_{i1},q_{i2})^{T}$, $i\in\{1,...,n\}$ do not form a regular polygon. Then there has to be a $j_{u}\in\{j_{1},...,j_{k}\}$ such that $\alpha_{j_{2}}-\alpha_{j_{1}}<\alpha_{j_{u+1}}-\alpha_{j_{u}}$, as otherwise we have that $\alpha_{j_{2}}-\alpha_{j_{1}}=\alpha_{j_{u+1}}-\alpha_{j_{u}}$ for all $j_{u}\in\{j_{1},...,j_{k}\}$, which then means that the $(q_{j1},q_{j2})^{T}$, $j\in\{1,...,n\}$ form a regular polygon with $k$ vertices. Let
    \begin{align*}
      C_{ij}(r)=\frac{\cos{(\alpha_{j}-\alpha_{i})}-\cos{(\beta_{j}-\beta_{i})}}{(1-(\cos{(\beta_{j}-\beta_{i})}+r^{2}(\cos{(\alpha_{j}-\alpha_{i})}-\cos{(\beta_{j}-\beta_{i})}))^{2})^{\frac{3}{2}}}.
    \end{align*}
    Then subtracting the identity obtained by taking $i=j_{1}$ in (\ref{Crit3}) from the identity obtained by taking $i=j_{2}$ in (\ref{Crit3}) we get that
    \begin{align}\label{Cosine equation}
      0=\left(\sum\limits_{s\in S_{\alpha j_{2}}}m_{s}-\sum\limits_{s\in S_{\alpha j_{1}}}m_{s}\right)C_{j_{1}j_{2}}(r)+\sum\limits_{v=3}^{k}\left(\sum\limits_{s\in S_{\alpha j_{v}}}m_{s}\right)\left(C_{j_{2}j_{v}}-C_{j_{1}j_{v}}\right).
    \end{align}
    Note that (\ref{Cosine equation}) does not involve sums $\sum\limits_{j\neq j_{1}, j\in S_{\alpha_{j_{1}}}}m_{j}C_{j_{1}j}$ and $\sum\limits_{j\neq j_{2}, j\in S_{\alpha_{j_{2}}}}m_{j}C_{j_{2}j}$, as for $\tilde j\neq j_{1}$, $\tilde j\in S_{\alpha_{j_{1}}}$ and $\widehat{j}\neq j_{2}$, $\widehat{j}\in S_{\alpha_{j_{2}}}$ we have that $\cos{(\alpha_{j_{1}}-\alpha_{\tilde j})}=1$ and $\cos{(\alpha_{j_{2}}-\alpha_{\widehat{j}})}=1$, so by Lemma~\ref{Lemma 2} we have that $C_{\tilde j j_{1}}$ and $C_{\widehat{j}j_{2}}$ are linearly independent from $C_{j_{1}j_{2}}$, $C_{j_{1}j_{v}}$, $C_{j_{2}j_{v}}$, $v\in\{3,...,k\}$ in (\ref{Cosine equation}). \\
    The only way for any term involving $C_{j_{1}j_{u}}$ in (\ref{Cosine equation}) to cancel out against a term $C_{j_{2}j_{\widehat{v}}}$ is by Lemma~\ref{Lemma 2} if $\cos{(\alpha_{j_{u}}-\alpha_{j_{1}})}=\cos{(\alpha_{j_{\widehat{v}}}-\alpha_{j_{2}})}$, which means that $\alpha_{j_{u}}-\alpha_{j_{1}}=\alpha_{j_{\widehat{v}}}-\alpha_{j_{2}}$, or $\alpha_{j_{u}}-\alpha_{j_{1}}=2\pi-(\alpha_{j_{\widehat{v}}}-\alpha_{j_{2}})$. If $\alpha_{j_{u}}-\alpha_{j_{1}}=\alpha_{j_{\widehat{v}}}-\alpha_{j_{2}}$, then $\alpha_{j_{\widehat{v}}}-\alpha_{j_{u}}=\alpha_{j_{2}}-\alpha_{j_{1}}<\alpha_{j_{u+1}}-\alpha_{j_{u}}$, giving $\alpha_{j_{\widehat{v}}}<\alpha_{j_{u+1}}$ and $\alpha_{j_{\widehat{v}}}-\alpha_{j_{u}}=\alpha_{j_{2}}-\alpha_{j_{1}}>0$, giving $\alpha_{j_{\widehat{v}}}>\alpha_{j_{u}}$ and therefore $\alpha_{j_{u}}<\alpha_{j_{\widehat{v}}}<\alpha_{j_{u+1}}$, which is impossible. So that means that for any term involving $C_{j_{1}j_{u}}$ in (\ref{Cosine equation}) to cancel out against a term $C_{j_{2}j_{\widehat{v}}}$ we have that $\alpha_{j_{u}}-\alpha_{j_{1}}=2\pi-(\alpha_{j_{\widehat{v}}}-\alpha_{j_{2}})$. Additionally there might be a $C_{j_{1}j_{w}}$, $j_{w}\in\{j_{1},...,j_{k}\}$ such that $\cos{(\alpha_{j_{u}}-\alpha_{j_{1}})}=\cos{(\alpha_{j_{w}}-\alpha_{j_{1}})}$. Let \begin{align*}V_{\alpha j_{\widehat{v}}}=\{s\in S_{\alpha j_{\widehat{v}}}|\cos{(\beta_{s}-\beta_{j_{2}})}=\cos{(\beta_{j_{\widehat{v}}}-\beta_{j_{2}})}\},\\ V_{\alpha j_{u}}=\{s\in S_{\alpha u}|\cos{(\beta_{s}-\beta_{j_{1}})}=\cos{(\beta_{j_{u}}-\beta_{j_{1}})}\}\end{align*} and $V_{\alpha j_{w}}=\{s\in S_{\alpha j_{w}}|\cos{(\beta_{w}-\beta_{j_{1}})}=\cos{(\beta_{j_{w}}-\beta_{j_{1}})}\}$. We then have by (\ref{Cosine equation}) that
    \begin{align}\label{MT1}0=\left(\sum\limits_{s\in V_{\alpha j_{\widehat{v}}}}m_{s}\right)-\left(\sum\limits_{s\in V_{\alpha j_{u}}}m_{s}\right)-\left(\sum\limits_{s\in V_{\alpha j_{w}}}m_{s}\right),\end{align}
    as $C_{j_{1}j_{u}}=C_{j_{2}j_{\widehat{v}}}=C_{j_{1}j_{w}}\neq 0$. Additionally, let
    \begin{align*}
      \widehat{S}_{ij}(r)=\frac{\sin{(\alpha_{j}-\alpha_{i})}}{(1-(\cos{(\beta_{j}-\beta_{i})}+r^{2}(\cos{(\alpha_{j}-\alpha_{i})}-\cos{(\beta_{j}-\beta_{i})}))^{2})^{\frac{3}{2}}}.
    \end{align*}
    Then by (\ref{Crit1}) we have that
    \begin{align}\label{MT3}0=\left(\sum\limits_{s\in V_{\alpha j_{\widehat{v}}}}m_{s}\right)\widehat{S}_{j_{2}j_{\widehat{v}}}.\end{align}
    If $\widehat{S}_{j_{1}j_{u}}=\widehat{S}_{j_{2}j_{\widehat{v}}}=0$, then $\sin{(\alpha_{j_{u}}-\alpha_{j_{1}})}=0$, which means that $\alpha_{j_{u}}-\alpha_{j_{1}}=0$ or $\alpha_{j_{u}}-\alpha_{j_{1}}=\pi$ and $\alpha_{j_{u}}-\alpha_{j_{1}}=2\pi-(\alpha_{j_{\widehat{v}}}-\alpha_{j_{2}})$, so $\alpha_{j_{u}}-\alpha_{j_{1}}=\pi$ and $\alpha_{j_{\widehat{v}}}-\alpha_{j_{2}}=\pi$, which contradicts the construction of $j_{u}$.
    We now therefore find the following: If there is a $\widehat{v}$ such that $\cos{(\alpha_{j_{u}}-\alpha_{j_{1}})}=\cos{(\alpha_{j_{\widehat{v}}}-\alpha_{j_{2}})}$, then we have by (\ref{MT3}) that $0\neq 0$, which is a contradiction. If there is no such $\widehat{v}$, then we have by (\ref{MT1}) that $0=0-\left(\sum\limits_{s\in V_{\alpha j_{u}}}m_{s}\right)-\left(\sum\limits_{s\in V_{\alpha j_{w}}}m_{s}\right)<0$, which is again a contradiction. This means that our initial assumption about the existence of $j_{u}$ is incorrect, which means that this proves that the configuration of the vectors $(q_{i1},q_{i2})^{T}$, $i\in\{1,...,n\}$, is a regular polygon and by repeating the argument so far after interchanging the roles of the $\alpha$s and $\beta$s, we find that the same is true for the configuration of the vectors $(q_{i3},q_{i4})^{T}$. This completes the proof.

\bibliographystyle{amsplain}

\end{document}